\documentclass{amsart}
\usepackage{graphicx}
\usepackage[usenames]{color}
\usepackage[normalem]{ulem}
\usepackage{cancel}

\newtheorem{thm}{Theorem}[section]
\newtheorem{cor}[thm]{Corollary}

\theoremstyle{definition}
\newtheorem{defn}[thm]{Definition}
\theoremstyle{remark}
\newtheorem{rem}[thm]{Remark}
\numberwithin{equation}{section}

\begin{document}
    \title{{Generalization of Lohwater-Pommerenke's Theorem}}%
    \author{P.V.Dovbush}


    \begin{abstract}
        In this paper, as an application of Zalcman's lemma in $\mathbb{C}^n$, we give a sufficient condition for normality of holomorphic functions of several complex variables, which generalizes previous known one-dimensional criterion of A.J.
        Lohwater and Ch.
        Pommerenke \cite[Theorem 1]{MR0338381}.
    \end{abstract}
    \maketitle
    A ``heuristic principle'' attributed to Andr\'{e} Bloch says that a family of holomorphic functions which have a property $P$ in common in a domain $\Omega\subset \mathbb{C}$ is [apt to be] a normal family in $\Omega$ if $P$ cannot be possessed by non-constant holomorphic functions in the whole plane $\mathbb{C}$.
    [An example of such a P is ``uniform boundedness.'']

    A rigorous formulation and proof of this was given in 1975 by Zalcman \cite{MR379852}.
    Zalcman's work was inspired by the result of Lohwater and Pommerenke \cite{MR0338381}.
    Their theorem deals with normal functions, not normal families, but the proofs are almost identical.

    It is the purpose of this note to give a generalization of the result of Lohwater and Pommerenke \cite{MR0338381} for normal functions defined on bounded domains of $\mathbb{C}^n$.

    It is known that the notion of normality can be generalized in various ways to higher dimensions.
    Here we adopt the definition of Cima and Krantz \cite[p. 305]{MR700143}.

    Let $\Omega$ be a bounded domain in $\mathbb{C}^n$.
    By $B(a,r)$ we denote the ball in $\mathbb{C}^n$ with center $a$ and radious $r$.
    Thus $B(a,r)$ consist of all $z\in \mathbb{C}^n$ such that $|z-a|<r$.

    For every function $\varphi$ of class $C^2(\Omega)$ we define at each point $z\in \Omega$ an Hermitian form
    \[
        L_z(\varphi, v):=\sum_{k,l=1}^n \frac{\partial^2\varphi}{\partial z_k \partial \overline{z}_l}(z) v_k \overline{v}_{l}
    \]
    and call it the Levi form of the function $\varphi$ at $z$.
    For a holomorphic function $f$ in $\Omega$, set
    \begin{equation}
        \label{e2}
        f^\sharp (z):=\sup_{ |v|=1}\sqrt{L_z(\log(1+|f|^2), v)}.
    \end{equation}
    This quantity $f^\sharp (z)$ is well defined since the Levi form $L_z(\log(1+|f|^2), v)$ is nonnegative for all $z\in \Omega$.

    Let $U$ be a unit disk in $\mathbb{C}$.
    The infinitesimal Kobayashi metric on $\Omega$ is given by
    \[
        K_\Omega(z, v)=\inf \{\alpha : \alpha>0 \textrm{ and }\exists g : \Omega\to U \textrm{ holomorphic, }g(0) = z \textrm{ and }g'(0) = v/\alpha\}
    \]
    \begin{defn}
        \label{def-norm}
        A holomorphic function $f \colon \Omega\to \mathbb{C}$ is called normal if exists a constant $C$, $0<C<\infty$, such that
        \begin{equation}
            \label{ch2:ineq:normfunc:7}
            L_z(\log (1+|f|^2),v)\leq C \cdot K_\Omega^2(z,v)
        \end{equation}
        for all $(z,v) \in \Omega\times \mathbb{C}^n$.
    \end{defn}

    \begin{thm}
        \label{LP1}
        A non-constant function $f$ holomorphic on $\Omega\subset \mathbb{C}^n$ is non-normal if there exist sequences $z_j\in \Omega$, $\rho_j=1/f^\sharp(z_j)\to 0$, such that the sequence
        \[
            g_j(\zeta)=f(z_j+\rho_j \zeta)
        \]
        converges locally uniformly in $\mathbb{C}^n$ to a non-constant entire function $g$ satisfying $g^\sharp(\zeta)\leq g^\sharp(0)=1$.
    \end{thm}
    \begin{proof}
        Let $\{p_j\}$ be an arbitrary sequence of points in $\Omega$, then $B(p_j, \delta_j)\subset \Omega$, where $\delta_j=dist(p_j,\partial\Omega)$.
        By the distance-decreasing property of Kobayashi metric
        \[
            K_\Omega(z,v)\leq K_{B(p_j, \delta_j)}(z,v)
        \]
        for all $(z,v)\in B(p_j, \delta_j)\times \mathbb{C}^n$.
        The Kobayashi metric of $B(p_j, \delta_j)$ is given by
        \[
            K_{B(p_j, \delta_j)}(z,v)=\frac{[(\delta_j^{2}-|z-p_j|^2)|v|^2+|(z-p_j,v)|^2]^{1/2}}{\delta_j^{2}-|z-p_j|^2}
        \]
        which clearly satisfy the inequality:
        \[
            K_{B(p_j, \delta_j)}(z,v)\leq \frac{\delta_j|v|}{\delta_j^{2}-|z-p_j|^2}.
        \]
        If $f$ is normal in $\Omega$ and then from (\ref{ch2:ineq:normfunc:7}) follows
        \begin{equation}
            \label{ch2:ineq:LP2}
            f^\sharp(p_j+\delta_j\zeta)\leq \frac{\sqrt{C}\delta_j}{\delta_j^{2}-|\delta_j\zeta|^2}.
        \end{equation}
        Set $g_j(\zeta):=f(p_j+\delta_j\zeta)$.
        By the invariance of the Levi form under biholomorphic mappings, we have
        \[
            L_\zeta(\log(1+|g_j|^2),v)=L_{p_j+\delta_j\zeta}(\log(1+|f_j|^2),\delta_{j}v)
        \]
        and so
        \[
            g_j^\sharp(\zeta)=\delta_{j}f^\sharp(p_j+\delta_j\zeta).
        \]
        It follows from (\ref{ch2:ineq:LP2}) that
        \[
            g_j^\sharp(\zeta)\leq \frac{\sqrt{C}}{1-|\zeta|^2}
        \]
        for all $j$ and all $\zeta\in B(0,1)$.
        By Marty's Criterion (Theorem \cite[Theorem 2.1]{MR4071476}) the family $\{g_j(\zeta)\}$ is normal in the unit ball $B(0,1)$.

        So if $f$ is not normal function in $\Omega$, then there exists a sequence $\{p_j\}$ in $\Omega$ such that $\{g_j(\zeta):=f(p_j+\delta_j\zeta)\}$ is not a normal sequence in a point, say, $\zeta_0$, $\zeta_0\in B(0,1)$.
        It follows from Zalcman's lemma \cite[Theorem 3.1]{MR4071476} that there exist $\zeta_j\to \zeta_0$, $\rho_j=1/g_j^\sharp(\zeta_j)\to 0$, such that the sequence
        \[
            g_j(\zeta)=f_j(p_j+\delta_j(\zeta_j+\rho_j \zeta))
        \]
        converges locally uniformly in $\mathbb{C}^n$ to a non-constant entire function $g$ satisfying $g^\sharp(\zeta)\leq g^\sharp(0)=1$.
        A simple calculation shows that $\delta_j\rho_j=1/f^\sharp(p_j+\delta_j\zeta_j)$ and therefore
        \[
            g_j(\zeta)=f_j(p_j+\delta_j\zeta_j+\zeta/f^\sharp(p_j+\delta_j\zeta_j))
        \]
        converges locally uniformly in $\mathbb{C}^n$ to a non-constant entire function $g$ satisfying $g^\sharp(\zeta)\leq g^\sharp(0)=1$.
        It follows $z_j=p_j+\delta_j\zeta_j$, $\rho_j=1/f^\sharp(p_j+\delta_j\zeta_j)$ do the work.
        This completes the proof of Theorem \ref{LP1}.
    \end{proof}

    The next result is closely related to the preceding theorem and is essentially a reformulation of (\ref{ch2:ineq:normfunc:7}).
    \begin{thm}
        \label{LP}
        Let $\Omega \subset \mathbb{C}^n$ be a bounded domain.
        If $f: \Omega\to \mathbb{C}$ is a normal holomorphic function, then for every choice of sequences $\{p_j\}$ in $\Omega$ and $\{r_j\}$, $r_j>0$, with $\lim_{j\to \infty}r_j/\delta_j = 0$, where $\delta_j=dist(p_j,\Omega)$, the sequence $\{f(p_j+r_j\zeta)\}$ converges locally uniformly to a constant function in $\mathbb{C}^n$.
    \end{thm}
    \begin{proof}
        Set $R_j = \delta_j/r_j$.
        It is clear that $R_j\to \infty$.
        Without restriction we can assume that $R_j>j$.
        Then for all $\zeta\in \mathbb{C}^n$ such that $|\zeta| < j$, we have
        \[
            |p_j+r_j\zeta-p_j|= r_j|\zeta|<\delta_{j}
        \]
        so that $p_j+r_j\zeta \in B(p_j, \delta_j)\subset \Omega$.
        Hence $g_j(\zeta):=f(p_j+r_j\zeta)$ is a holomorphic function on the ball $B(0,j)=\{\zeta\in \mathbb{C}^n : |\zeta|<j\}$.

        It is an immediate consequence of the definition that since $f$ is the normal function, then there exists a positive constant $C$ such that
        \[
            L_z(\log (1+|f|^2),v)\leq C \cdot K_\Omega^2(z,v)
        \]
        for all $(z,v)\in \Omega\times \mathbb{C}^n$.

        Since $B(p_j, \delta_j)$ is contained in $\Omega$ the distance-decreasing property yields
        \[
            K_\Omega(z,v)\leq K_{B(p_j, \delta_j)}(z,v)
        \]
        for all $(z,v)\in B(p_j, \delta_j)\times \mathbb{C}^n$.
        Since
        \[
            K_{B(p_j, \delta_j)}(z,v)=\frac{[(\delta_j^{2}-|z-p_j|^2)|v|^2+|(z-p_j,v)|^2]^{1/2}}{\delta_j^{2}-|z-p_j|^2}.
        \]
        we have
        \[
            K_{B(p_j, \delta_j)}(z,v)\leq \frac{\delta_j|v|}{\delta_j^{2}-|z-p_j|^2}.
        \]
        Hence
        \[
            K_\Omega(z,v)\leq \frac{\delta_j|v|}{\delta_j^{2}-|z-p_j|^2}
        \]
        for all $(z,v)\in B(p_j, \delta_j)\times \mathbb{C}^n$.
        Therefore,
        \[
            \sqrt{L_{p_j+r_j\zeta}(\log(1+|f_j|^2),v)}\leq \frac{\sqrt{C}\delta_j|v|}{\delta_j^{2}-|r_j\zeta|^2}
        \]
        for all $(\zeta,v)\in B(0,j)\times \mathbb{C}^n$.
        Taking $\sup$ on both sides over $|v|= 1$, we have
        \begin{equation}
            \label{ch2:ineq:LP0}
            f^\sharp(p_j+r_j\zeta)\leq \frac{\sqrt{C}\delta_j}{\delta_j^{2}-|r_j\zeta|^2}
        \end{equation}
        By the invariance of the Levi form under biholomorphic mappings, we have
        \[
            L_\zeta(\log(1+|g_j|^2),v)=L_{p_j+r_j\zeta}(\log(1+|f_j|^2),r_{j}v)
        \]
        and hence
        \begin{equation}
            \label{ch2:ineq:LP1}
            g_j^\sharp (\zeta)=r_{j}f^\sharp(p_j+r_j\zeta).
        \end{equation}
        We note from (\ref{ch2:ineq:LP0}), (\ref{ch2:ineq:LP1}), and $\delta_j/r_j>j$, that
        \[
            g_j^\sharp (\zeta)\leq \frac{\sqrt{C}r_j\delta_j}{\delta_j^{2}-|r_j\zeta|^2}\leq \frac{\sqrt{C}/j}{1-(1/j)^2|\zeta|^2}
        \]
        for all $j$ sufficiently large and all $\zeta$, $|\zeta|<j$.

        For every $m \in N$ the sequence $\{g_j\}_{j>m}$ is normal in $|\zeta| < m$ by Marty's Theorem (Theorem \cite[Theorem 2.1]{MR4071476}).
        The well-known Cantor diagonal process yields a subsequence $\{g_k=g_{j_k}\}$ which converges uniformly on every ball $|\zeta| < R$.
        The limit function $g$ is holomorphic and satisfies $g^\sharp (\zeta) = 0$ which yields: $dg(\zeta) = 0$ for all $\zeta\in \mathbb{C}^n$, i.e.
        $g(\zeta)\equiv \textrm{constant}$ in $\mathbb{C}^n$.
    \end{proof}

    Theorem \ref{LP} can be restated in the following way.

    \begin{cor}
        \label{CLP}
        Let $\Omega\subset \mathbb{C}^n$ be a bounded domain.
        If $f: \Omega\to \mathbb{C}$ is a holomorphic function, and there exist sequences $\{p_j\}$ in $\Omega$ and $\{r_j\}$, $r_j>0$, with $\lim_{j\to \infty}r_j/\delta_j = 0$, where $\delta_j=\delta_\Omega(p_j)$, such that $\{f(p_j+r_j\zeta)\}$ converges locally uniformly to a non-constant holomorphic function in $\mathbb{C}^n$, then $f$ is non-normal.
    \end{cor}

    \begin{rem}
        In \cite{MR871724}, Theorem \ref{LP} was proven for the case of the unit ball in $\mathbb{C}^n$.
        The converse to Theorem \ref{LP} may not hold as it stands for $n\geq 2$.
        Lohwater and Pommerenke \cite[Theorem 1]{MR0338381} originally stated their theorem with no restriction on the speed at which $\rho_n\to 0$.
        In proving their theorem they asserted, "if $f$ is normal and $f(z_n + \rho_n\zeta) \to g(\zeta)$ locally uniformly, then $\rho_n/(1 - |z_n|) \to 0$".
        The statement in quotes is false as one can see from $f(z)=z$, $z_n=1-n^{-3}$, $\rho_n=n^{-2}$, $g(\zeta)\equiv 1$.

        This is one of the reason that the true generalization of Lohwater-Pommerenke's theorem to the higher-dimensional domain most likely breaks down.
    \end{rem}

    
\end{document}